\documentclass[12pt,oneside]{amsart}
\usepackage[T1]{fontenc}
\usepackage[latin9]{inputenc}
\usepackage{mathtools}
\usepackage{amsbsy}
\usepackage{amstext}
\usepackage{amsthm}
\usepackage{amssymb}

\makeatletter
\numberwithin{equation}{section}
\numberwithin{figure}{section}
  \theoremstyle{plain}
  \newtheorem*{conjecture*}{\protect\conjecturename}
\theoremstyle{plain}
\newtheorem{thm}{\protect\theoremname}[section]
  \theoremstyle{plain}
  \newtheorem*{thm*}{\protect\theoremname}
  \theoremstyle{remark}
  \newtheorem{rem}[thm]{\protect\remarkname}
  \theoremstyle{plain}
  \newtheorem{prop}[thm]{\protect\propositionname}
  \theoremstyle{plain}
  \newtheorem{lem}[thm]{\protect\lemmaname}

\usepackage{dsfont}
\usepackage{pgf,tikz}
\usepackage{mathrsfs}

\usetikzlibrary{arrows}

\definecolor{xdxdff}{rgb}{0.49,0.49,1.}
\definecolor{qqqqff}{rgb}{0.,0.,1.}
\definecolor{ffqqqq}{rgb}{1.,0.,0.}

\makeatother

  \providecommand{\conjecturename}{Conjecture}
  \providecommand{\lemmaname}{Lemma}
  \providecommand{\propositionname}{Proposition}
  \providecommand{\remarkname}{Remark}
  \providecommand{\theoremname}{Theorem}
\providecommand{\theoremname}{Theorem}

\begin{document}

\title[Bunkbed conjecture on the complete graph for $p\geqslant1/2$]{A proof of the Bunkbed conjecture on the complete graph for $p\geqslant1/2$}

\author{Paul de Buyer}
\address[P. de Buyer]{Université Paris Nanterre - Modal'X, 200 avenue
de la République 92000 Nanterre, France}
\email{debuyer@math.cnrs.fr}



\begin{abstract}
The bunkbed of a graph $G$ is the graph $G\times\left\{ 0,1\right\} $.
It has been conjectured that in the independent bond percolation model,
the probability for $\left(u,0\right)$ to be connected with $\left(v,0\right)$
is greater than the probability for $\left(u,0\right)$ to be connected
with $\left(v,1\right)$, for any vertex $u$, $v$ of $G$. In this
article, we prove this conjecture for the complete graph in the case
of the independent bond percolation of parameter $p\geqslant1/2$. 
\end{abstract}

\maketitle

\section{Introduction}

Percolation theory has been widely studied over the last decades and
yet, several intuitive results are very hard to prove rigorously.
This is the case of the bunkbed conjecture formulated by Kasteleyn
(published as a remark in \cite{vdBKJCorrelation}) which investigates
a notion of graph distance through percolation theory. 

Consider a graph $G=\left(V,E\right)$ where $E$ is the set of edges
and delete each edge independently with probability $p$; we obtain
a random graph which law is noted $\mathbb{P}_{p}$. The main question
in percolation is to understand and bound the probability of two set
of vertices to be connected (for a general introduction on the subject,
see \cite{Grimmett}). Furthermore, one can properly build a distance
$d:V\times V$ between vertices using percolation, saying that for
three vertices $u,v,w\in V$, $v$ is further from $u$ than $w$
from $u$ if the probability for $u$ and $v$ to be connected is
smaller than the probability for $u$ and $w$ to be connected, and
ask if this distance coincide with the usual graph distance $d_{G}$.
However, the validity of this property depends on the graph and the
value of $p$, see figure $\ref{fig:Counter-example}$. 

\begin{figure}[!h]
\begin{center}
\begin{tikzpicture}[scale=0.75, line cap=round,line join=round,>=triangle 45,x=1cm,y=1cm]\clip(2,1) rectangle (10,8); 
	\draw [line width=1pt] (5,4)-- (6,6); 
	\draw [line width=1pt] (6,6)-- (6,4);
	\draw [line width=1pt] (6,4)-- (6,2);
	\draw [line width=1pt] (6,2)-- (5,4); 
	\draw [line width=1pt] (4,4)-- (6,2);
	\draw [line width=1pt] (6,2)-- (7,4);
	\draw [line width=1pt] (7,4)-- (6,6); 
	\draw [line width=1pt] (6,6)-- (8,4); 	
	\draw [line width=1pt] (8,4)-- (6,2);
	\draw [line width=1pt] (4,4)-- (6,6); 
	\begin{scriptsize}
		\draw [fill=black] (4,4) circle (3pt); 		
		\draw [fill=black] (6,4) circle (3pt); 
		\draw [fill=black] (6,6) circle (3pt);
		\draw[color=black] (6.8,6.1) node {\large $w$};
		\draw [fill=black] (6,2) circle (3pt);
		\draw[color=black] (6.7,2) node {\large $u$};
		\draw [fill=black] (5,4) circle (3pt);
		\draw [fill=black] (7,4) circle (3pt);
		\draw [fill=black] (8,4) circle (3pt);
		\draw[color=black] (8.6,4) node {\large $v$};
	\end{scriptsize} 
\end{tikzpicture}
\end{center}

\caption{\label{fig:Counter-example}Counter-example where the graph distance
may not coincide with the percolation distance}

\end{figure}
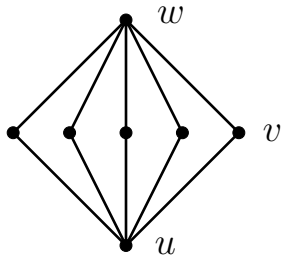

Before recalling the bunkbed conjecture, we give another problem closely
related. Consider the centered box of $B_{n}=\left[-n;n\right]^{d}\cap\mathbb{Z}^{d}$
and its border $\partial B_{n}=\left\{ x\in B_{n}:\left\Vert x\right\Vert _{\infty}=n\right\} $,
and consider the following function $f\colon x\mapsto\mathbb{P}_{p}\left(x\leftrightarrow\partial B_{n}\right)$.
An open question of interest is to ask if the minimum of the function
$f$ is achieved at the origin, \emph{i.e.} if $\min_{x}f\left(x\right)=f\left(0\right)$.
One can extend the question in the following way: consider $x$ and
$y$ two vertices of $B_{n}$ such that $d_{G}\left(x,\partial B_{n}\right)<d_{G}\left(y,\partial B_{n}\right)$,
is it true that $f\left(x\right)\geqslant f\left(y\right)$? To the
best knowledge of the author, these questions remain open. In order
to understand these problems, some structure has been added to the
graph giving rise to the study of the bunkbed conjecture defined as
follow.

A bunkbed graph $G=(V,E)$ of a graph $\widetilde{G}=(\widetilde{V},\widetilde{E})$
is the graph given by the superposition of two identical copy of $\widetilde{G}$
to which we add the edges connecting the symetrical vertices, see
figure $\ref{fig:BBCFigureQuiDechire}$; informally, we note $G=\widetilde{G}\times\left\{ \text{0;1}\right\} $.
Furthermore, we say that a vertex $u$ belongs to the bottom graph
if $u=\left(x,0\right)$, $u$ belongs to the top graph if $u=\left(x,1\right)$,
and $u'$ is the symmetrical of $u$ if $u=\left(x,i\right)$ and
$u'=\left(x,1-i\right)$. The statement of the conjecture is the following:
\begin{conjecture*}
Let $G$ be a bunkbed graph. Let $u$ and $v$ be two vertices of
the bottom graph, and $v'$ the symmetrical of $v$, then for any
$p\in\left[0;1\right]$ the following holds:
\[
\mathbb{P}_{p}\left(u\leftrightarrow v\right)\geqslant\mathbb{P}_{p}\left(u\leftrightarrow v'\right)
\]
\end{conjecture*}
Litterature about the conjecture is fairly poor even if the problem
aroused the interest of quite a number of researchers. The work of
S. Linusson and M. Leander \cite{leandersjalvstandiga,linusson2011percolation}
who proved that the conjecture holds for a subclass of the planar
graphs, called outerplanar graphs and wheel graphs. The proof used
in their paper, called the minimal counter-example, is combinatoric
and relies on the structure of the graph making it difficult to extend
to a more general class of graph. 

O. Häggström proved in \cite{haggstrom2003probability} a similar
conjecture for the Ising model on a general graph. Recall that the
Ising model assign a value $1$ or $-1$ to each vertex accordind
to a Gibbs measure, and it has been shown that the value of $u$ has
more influence on the value of $v$ than $v'$ in the sense that $\mathbb{E}\left[\omega\left(u\right)\omega\left(v\right)\right]\geqslant\mathbb{E}\left[\omega\left(u\right)\omega\left(v'\right)\right]$
where $\omega$ is a configuration in $\left\{ -1;1\right\} ^{V}$
and $\omega\left(u\right)$ is the value assigned to the vertex $u$.
The proof relies on the link between FK-percolation and Ising model
as well as some properties of the Gibbs measure which can not extend
to the independent case.

Related to the study of this conjecture, a type of Harris-FKG inequality
conditionned by a decreasing event has been proven in \cite{vdBHK2006,vdBKJCorrelation}.

Furthermore, in the random walk field, an equivalent conjecture of
the bunkbed conjecture formulated in \cite{bollobas1997random} is
that starting from a vertex $u$, the first time of reaching a vertex
$v$ is shorter (in some sense) that the first time of reaching the
vertex $v'$ which has been proved in \cite{haggstrom1998conjecture}.

Finally, in the article \cite{RCequivalent} of Rudzinski and Smyth,
an extensive list of equivalent reformulation of the Bunkbed Conjecture
is presented.

In this article, in the context of independent bond percolation model,
we prove the following theorem:
\begin{thm}
\label{thm:BBCTheoremePrincipal}The bunkbed conjecture is verified
for the complete graph with $n\geqslant1$ vertices when the paramater
of percolation $p\geqslant1/2$.
\end{thm}

The paper is organized in the following way. In the second section,
we introduce formally all the notations and the results. In the third
section, we give the proof of the main theorem which uses two technical
lemmas proven in the fourth section. In the fifth section, we prove
the secondary results, give deeper explanations the consequences of
the proof, and possible leads to solve the conjecture.

\section{Notations et results}

In this section, we introduce formally the model, the notations and
the main result. We recall that the complete graph with $n$ vertices
is noted $K_{n}$.

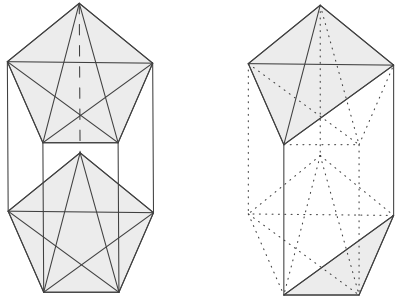
\begin{figure}[!h]
\begin{center}

\definecolor{uququq}{rgb}{0.25,0.25,0.25}  
\begin{tikzpicture}[line cap=round,scale=0.5, line join=round,>=triangle 45,x=1.0cm,y=1.0cm] \clip(15.,1.) rectangle (29,10.6);

\fill[color=uququq,fill=uququq,fill opacity=0.1] (18.94567,9.68891) -- (17.03361,8.135374) -- (17.97830,5.97830) -- (19.97830,5.97830) -- (20.89755,8.09554) -- cycle; 
\fill[color=uququq,fill=uququq,fill opacity=0.1] (18.96736386462253,5.710609775423925) -- (17.055312521904767,4.157068059465733) -- (18.,2.) -- (20.,2.) -- (20.919249610313596,4.117233656492446) -- cycle;
\fill[color=uququq,fill=uququq,fill opacity=0.1] (25.352206533612073,9.636951268385594) -- (23.44015519089431,8.083409552427405) -- (24.384842668989542,5.926341492961671) -- (27.30409227930314,8.04357514945412) -- cycle;
\fill[color=uququq,fill=uququq,fill opacity=0.1] (27.30409227930314,4.043575149454118) -- (24.384842668989542,1.9263414929616722) -- (26.384842668989542,1.9263414929616722) -- cycle;

\draw [color=uququq] (4.9673638646225315,5.710609775423926)-- (3.055312521904767,4.157068059465734); \draw [color=uququq] (3.055312521904767,4.157068059465734)-- (4.,2.); 
\draw [color=uququq] (4.,2.)-- (6.,2.); 
\draw [color=uququq] (6.,2.)-- (6.919249610313596,4.117233656492448); 
\draw [color=uququq] (4.9673638646225315,5.710609775423926)-- (6.919249610313596,4.117233656492448); 
\draw [color=uququq] (4.9673638646225315,5.710609775423926)-- (4.,2.);
\draw [color=uququq] (4.,2.)-- (6.919249610313596,4.117233656492448); 
\draw [color=uququq] (6.919249610313596,4.117233656492448)-- (3.055312521904767,4.157068059465734);
\draw [color=uququq] (3.055312521904767,4.157068059465734)-- (6.,2.); 
\draw [color=uququq] (4.9673638646225315,5.710609775423926)-- (6.,2.); 
\draw [color=uququq] (4.945670332138879,9.68891624294028)-- (3.0336189894211154,8.135374526982085); 
\draw [color=uququq] (3.0336189894211154,8.135374526982085)-- (3.978306467516348,5.978306467516349); 
\draw [color=uququq] (3.978306467516348,5.978306467516349)-- (5.978306467516348,5.978306467516349); 
\draw [color=uququq] (5.978306467516348,5.978306467516349)-- (6.897556077829944,8.095540124008801); 
\draw [color=uququq] (4.945670332138879,9.68891624294028)-- (6.897556077829944,8.095540124008801); 
\draw [color=uququq] (4.945670332138879,9.68891624294028)-- (3.978306467516348,5.978306467516349); 
\draw [color=uququq] (3.978306467516348,5.978306467516349)-- (6.897556077829944,8.095540124008801); 
\draw [color=uququq] (6.897556077829944,8.095540124008801)-- (3.0336189894211154,8.135374526982085); 
\draw [color=uququq] (3.0336189894211154,8.135374526982085)-- (5.978306467516348,5.978306467516349); 
\draw [color=uququq] (4.945670332138879,9.68891624294028)-- (5.978306467516348,5.978306467516349); 
\draw [color=uququq] (3.0336189894211154,8.135374526982085)-- (3.055312521904767,4.157068059465734); 
\draw [color=uququq] (3.978306467516348,5.978306467516349)-- (4.,2.); 
\draw [color=uququq] (5.978306467516348,5.978306467516349)-- (6.,2.); 
\draw [color=uququq] (6.897556077829944,8.095540124008801)-- (6.919249610313596,4.117233656492448); 
\draw [dash pattern=on 4pt off 4pt,color=uququq] (4.945670332138879,9.68891624294028)-- (4.9673638646225315,5.710609775423926); 
\draw [dotted,color=uququq] (11.352206533612074,5.636951268385596)-- (9.440155190894309,4.0834095524274066); 
\draw [dotted,color=uququq] (9.440155190894309,4.0834095524274066)-- (10.384842668989542,1.9263414929616722); 
\draw [color=uququq] (10.384842668989542,1.9263414929616722)-- (12.384842668989542,1.9263414929616722);
\draw [color=uququq] (12.384842668989542,1.9263414929616722)-- (13.30409227930314,4.043575149454119); 
\draw [dotted,color=uququq] (11.352206533612074,5.636951268385596)-- (13.30409227930314,4.043575149454119); 
\draw [dotted,color=uququq] (11.352206533612074,5.636951268385596)-- (10.384842668989542,1.9263414929616722); 
\draw [color=uququq] (10.384842668989542,1.9263414929616722)-- (13.30409227930314,4.043575149454119);
\draw [dotted,color=uququq] (13.30409227930314,4.043575149454119)-- (9.440155190894309,4.0834095524274066); 
\draw [dotted,color=uququq] (9.440155190894309,4.0834095524274066)-- (12.384842668989542,1.9263414929616722); 
\draw [dotted,color=uququq] (11.352206533612074,5.636951268385596)-- (12.384842668989542,1.9263414929616722); 
\draw [color=uququq] (11.35220653361207,9.636951268385594)-- (9.440155190894309,8.083409552427405);
\draw [color=uququq] (9.440155190894309,8.083409552427405)-- (10.384842668989542,5.926341492961672); 
\draw [dotted,color=uququq] (10.384842668989542,5.926341492961672)-- (12.384842668989542,5.926341492961672); 
\draw [dotted,color=uququq] (12.384842668989542,5.926341492961672)-- (13.30409227930314,8.04357514945412);
\draw [color=uququq] (11.35220653361207,9.636951268385594)-- (13.30409227930314,8.04357514945412); 
\draw [color=uququq] (11.35220653361207,9.636951268385594)-- (10.384842668989542,5.926341492961672); 
\draw [color=uququq] (10.384842668989542,5.926341492961672)-- (13.30409227930314,8.04357514945412); 
\draw [color=uququq] (13.30409227930314,8.04357514945412)-- (9.440155190894309,8.083409552427405); 
\draw [dotted,color=uququq] (9.440155190894309,8.083409552427405)-- (12.384842668989542,5.926341492961672); 
\draw [dotted,color=uququq] (11.35220653361207,9.636951268385594)-- (12.384842668989542,5.926341492961672); 
\draw [dotted,color=uququq] (9.440155190894309,8.083409552427405)-- (9.440155190894309,4.0834095524274066); 
\draw [color=uququq] (10.384842668989542,5.926341492961672)-- (10.384842668989542,1.9263414929616722);
\draw [dotted,color=uququq] (12.384842668989542,5.926341492961672)-- (12.384842668989542,1.9263414929616722); 
\draw [color=uququq] (13.30409227930314,8.04357514945412)-- (13.30409227930314,4.043575149454119); 
\draw [dotted,color=uququq] (11.35220653361207,9.636951268385594)-- (11.352206533612074,5.636951268385596); 
\draw [color=uququq] (18.96736386462253,5.710609775423925)-- (17.055312521904767,4.157068059465733);
\draw [color=uququq] (17.055312521904767,4.157068059465733)-- (18.,2.); \draw [color=uququq] (18.,2.)-- (20.,2.); 
\draw [color=uququq] (20.,2.)-- (20.919249610313596,4.117233656492446); 
\draw [color=uququq] (18.96736386462253,5.710609775423925)-- (20.919249610313596,4.117233656492446); 
\draw [color=uququq] (18.96736386462253,5.710609775423925)-- (18.,2.); 
\draw [color=uququq] (18.,2.)-- (20.919249610313596,4.117233656492446); 
\draw [color=uququq] (20.919249610313596,4.117233656492446)-- (17.055312521904767,4.157068059465733);
\draw [color=uququq] (17.055312521904767,4.157068059465733)-- (20.,2.); 
\draw [color=uququq] (18.96736386462253,5.710609775423925)-- (20.,2.); 
\draw [color=uququq] (18.945670332138878,9.68891624294028)-- (17.033618989421115,8.135374526982087); 
\draw [color=uququq] (17.033618989421115,8.135374526982087)-- (17.978306467516347,5.978306467516348); 
\draw [color=uququq] (17.978306467516347,5.978306467516348)-- (19.978306467516347,5.978306467516348); 
\draw [color=uququq] (19.978306467516347,5.978306467516348)-- (20.897556077829943,8.095540124008803); 
\draw [color=uququq] (18.945670332138878,9.68891624294028)-- (20.897556077829943,8.095540124008803); 
\draw [color=uququq] (18.945670332138878,9.68891624294028)-- (17.978306467516347,5.978306467516348); 
\draw [color=uququq] (17.978306467516347,5.978306467516348)-- (20.897556077829943,8.095540124008803); 
\draw [color=uququq] (20.897556077829943,8.095540124008803)-- (17.033618989421115,8.135374526982087); 
\draw [color=uququq] (17.033618989421115,8.135374526982087)-- (19.978306467516347,5.978306467516348); 
\draw [color=uququq] (18.945670332138878,9.68891624294028)-- (19.978306467516347,5.978306467516348); 
\draw [color=uququq] (17.033618989421115,8.135374526982087)-- (17.055312521904767,4.157068059465733);
\draw [color=uququq] (17.978306467516347,5.978306467516348)-- (18.,2.); 
\draw [color=uququq] (19.978306467516347,5.978306467516348)-- (20.,2.); 
\draw [color=uququq] (20.897556077829943,8.095540124008803)-- (20.919249610313596,4.117233656492446);
\draw [dash pattern=on 4pt off 4pt,color=uququq] (18.945670332138878,9.68891624294028)-- (18.96736386462253,5.710609775423925); 
\draw [dotted,color=uququq] (25.352206533612073,5.636951268385595)-- (23.44015519089431,4.083409552427406); 
\draw [dotted,color=uququq] (23.44015519089431,4.083409552427406)-- (24.384842668989542,1.9263414929616722);
\draw [color=uququq] (24.384842668989542,1.9263414929616722)-- (26.384842668989542,1.9263414929616722);
\draw [color=uququq] (26.384842668989542,1.9263414929616722)-- (27.30409227930314,4.043575149454118);
\draw [dotted,color=uququq] (25.352206533612073,5.636951268385595)-- (27.30409227930314,4.043575149454118); 
\draw [dotted,color=uququq] (25.352206533612073,5.636951268385595)-- (24.384842668989542,1.9263414929616722); 
\draw [color=uququq] (24.384842668989542,1.9263414929616722)-- (27.30409227930314,4.043575149454118); 
\draw [dotted,color=uququq] (27.30409227930314,4.043575149454118)-- (23.44015519089431,4.083409552427406); 
\draw [dotted,color=uququq] (23.44015519089431,4.083409552427406)-- (26.384842668989542,1.9263414929616722); 
\draw [dotted,color=uququq] (25.352206533612073,5.636951268385595)-- (26.384842668989542,1.9263414929616722); 
\draw [color=uququq] (25.352206533612073,9.636951268385594)-- (23.44015519089431,8.083409552427405); 
\draw [color=uququq] (23.44015519089431,8.083409552427405)-- (24.384842668989542,5.926341492961671);
\draw [dotted,color=uququq] (24.384842668989542,5.926341492961671)-- (26.384842668989542,5.926341492961671); 
\draw [dotted,color=uququq] (26.384842668989542,5.926341492961671)-- (27.30409227930314,8.04357514945412); 
\draw [color=uququq] (25.352206533612073,9.636951268385594)-- (27.30409227930314,8.04357514945412); 
\draw [color=uququq] (25.352206533612073,9.636951268385594)-- (24.384842668989542,5.926341492961671); 
\draw [color=uququq] (24.384842668989542,5.926341492961671)-- (27.30409227930314,8.04357514945412); 
\draw [color=uququq] (27.30409227930314,8.04357514945412)-- (23.44015519089431,8.083409552427405); 
\draw [dotted,color=uququq] (23.44015519089431,8.083409552427405)-- (26.384842668989542,5.926341492961671);
\draw [dotted,color=uququq] (25.352206533612073,9.636951268385594)-- (26.384842668989542,5.926341492961671); 
\draw [dotted,color=uququq] (23.44015519089431,8.083409552427405)-- (23.44015519089431,4.083409552427406); 
\draw [color=uququq] (24.384842668989542,5.926341492961671)-- (24.384842668989542,1.9263414929616722); 
\draw [dotted,color=uququq] (26.384842668989542,5.926341492961671)-- (26.384842668989542,1.9263414929616722); 
\draw [color=uququq] (27.30409227930314,8.04357514945412)-- (27.30409227930314,4.043575149454118); 
\draw [dotted,color=uququq] (25.352206533612073,9.636951268385594)-- (25.352206533612073,5.636951268385595); 
\draw [color=uququq] (18.945670332138878,9.68891624294028)-- (17.033618989421115,8.135374526982087); 
\draw [color=uququq] (17.033618989421115,8.135374526982087)-- (17.978306467516347,5.978306467516348); 
\draw [color=uququq] (17.978306467516347,5.978306467516348)-- (19.978306467516347,5.978306467516348); 
\draw [color=uququq] (19.978306467516347,5.978306467516348)-- (20.897556077829943,8.095540124008803); 
\draw [color=uququq] (20.897556077829943,8.095540124008803)-- (18.945670332138878,9.68891624294028); 
\draw [color=uququq] (18.96736386462253,5.710609775423925)-- (17.055312521904767,4.157068059465733); 
\draw [color=uququq] (17.055312521904767,4.157068059465733)-- (18.,2.);
\draw [color=uququq] (18.,2.)-- (20.,2.); 
\draw [color=uququq] (20.,2.)-- (20.919249610313596,4.117233656492446); 
\draw [color=uququq] (20.919249610313596,4.117233656492446)-- (18.96736386462253,5.710609775423925); 
\draw [color=uququq] (25.352206533612073,9.636951268385594)-- (23.44015519089431,8.083409552427405); 
\draw [color=uququq] (23.44015519089431,8.083409552427405)-- (24.384842668989542,5.926341492961671); 
\draw [color=uququq] (24.384842668989542,5.926341492961671)-- (27.30409227930314,8.04357514945412); 
\draw [color=uququq] (27.30409227930314,8.04357514945412)-- (25.352206533612073,9.636951268385594); 
\draw [color=uququq] (27.30409227930314,4.043575149454118)-- (24.384842668989542,1.9263414929616722); 
\draw [color=uququq] (24.384842668989542,1.9263414929616722)-- (26.384842668989542,1.9263414929616722); 
\draw [color=uququq] (26.384842668989542,1.9263414929616722)-- (27.30409227930314,4.043575149454118); 
\end{tikzpicture}\caption{Bunkbed graph of $K_{5}$ and an element of $G_{3,4,2}$}
\label{fig:BBCFigureQuiDechire}

\end{center}
\end{figure}

We call the bunkbed graph $G=\left(V,E\right)$ of a graph $\widetilde{G}=\left(\widetilde{V},\widetilde{E}\right)$
(called the original graph) is defined by: 
\begin{eqnarray*}
V & = & \tilde{V}\times\left\{ 0,1\right\} \\
E & = & \left\{ \left\{ \left(x,0\right),\left(y,0\right)\right\} :\left\{ x,y\right\} \in\tilde{E}\right\} \cup\left\{ \left\{ \left(x,1\right),\left(y,1\right)\right\} :\left\{ x,y\right\} \in\tilde{E}\right\} \\
 &  & \phantom{\left\{ \left\{ \left(x,0\right),\left(y,0\right)\right\} :\left\{ x,y\right\} \in\tilde{E}\right\} }\cup\left\{ \left\{ \left(x,0\right),\left(x,1\right)\right\} \right\} 
\end{eqnarray*}
An example is given in the figure $\ref{fig:BBCFigureQuiDechire}$.
As previously explained in the introduction, it is natural to distinguish
the vertices in the bottom graph, the set of vertices which can be
written $\left(x,0\right)$, and the vertices in the top graph, the
set of $\left(x,1\right)$. We call the symmetrical of a vertex $u=\left(x,i\right)$,
the vertex $u'=\left(x,1-i\right)$. In the rest of the article, letters
$u$ and $v$ will designate vertices of the bottom graph.

The percolation model is defined as follow. Consider a graph $G=\left(V,E\right)$
where $V$ is the set of vertices and $E$ the set of unoriented egdes.
We open each edges of $E$ independently with probability $p$ and
close them with probability $1-p$ and we write $\mathbb{P}_{p}$
the law associated to this percolation model. We call a configuration,
an element $\omega=\left(\omega_{e}\right)_{e\in E}\in\left\{ 0,1\right\} ^{E}$
corresponding to the bond percolation model where $\omega_{e}=0$
means that the edge $e$ is closed and $1$ means that the edge $e$
is open. We call a path between vertices $u$ and $v$ the set $\gamma$
of edges $\gamma=\left\{ e_{1},...,e_{n}\in E;e_{1}=\left(u,x_{1}\right),e_{2}=\left(x_{2},x_{3}\right),\ldots,e_{n}=\left(x_{n-1},v\right)\right\} $.
For a configuration $\omega$, we call an open path a path of open
edges of $\omega$. For two vertices $x,y\in V$, we write $x\overset{\omega}{\leftrightarrow}y$
if there exists an open path between $x$ and $y$. When there is
no ambiguity, we will omit the dependence of $\omega$ and write $x\leftrightarrow y$
instead of $x\overset{\omega}{\leftrightarrow}y$. By convention,
for any configuration, a vertex is always connected to itself, i.e.
$x\leftrightarrow x$. We recall our main theorem given in the introduction.
\begin{thm*}
\textbf{$\ref{thm:BBCTheoremePrincipal}$} Let $G=\left(V,E\right)$
be the bunkbed graph of $K_{n}$ with $n\geqslant1$, then for all
vertices of the bottom graph $u,v\in V$ and all $p\geqslant1/2$:
\[
\mathbb{P}_{p}\left(u\leftrightarrow v\right)\geqslant\mathbb{P}_{p}\left(u\leftrightarrow v'\right)
\]
\end{thm*}
\begin{rem}
We chose to study the complete graph since calculations are easier.
Furthermore, a way to solve the conjecture would be to prove that
if the conjecture is verified for a graphe $G$, then it is verified
for the graph $G\backslash\left\{ e,e'\right\} $ where we removed
an edge $e$ and it symmetric $e'$.
\end{rem}

\begin{rem}
Among the trivialities around the conjecture, for any graph $G$,
there exists a constant $p_{G}\in\left(0;1\right]$ such that the
conjecture is verified for all $p\leqslant p_{G}$. Indeed, when the
percolation parameter is small enough, only the shortest paths can
connect two vertices according to $\mathbb{P}_{p}$, the other paths
contributing negligibly. Since the shortest path between $u$ and
$v$ has a shorter length of 1 compared to the shortest path between
$u$ and $v'$, the conjecture is proved.
\end{rem}

As an auxiliary result, we prove that the conjecture holds in mean
in a more general setting, suggesting that the conjecture should be
true. Instead of keeping each edge with probability $p$, we keep
the edge $e$ with probability $p_{e}$, and we define a vector of
percolation parameter $\boldsymbol{p}=\left(p_{e}\right)_{e\in E}$.
In the context of the bunkbed conjecture, we say that the vector of
percolation parameter is constrained if an edge $e=\left\{ x,y\right\} $
has the same percolation parameter as its symmetrical $e'=\left\{ x',y'\right\} $,
\emph{i.e.} for all $e\in E$, $p_{e}=p_{e'}$.
\begin{prop}
\label{prop:BunkbedMoyenne}Let $G$ be a bunkbed graph. Let $X$
and $Y$ be two random variables independent and identically distributed
on the vertices of the bottom graph of $G$ according to a law $P$.
Then for any vector of percolation parameter $\boldsymbol{p}$ constrained,
the following holds:
\[
E\left[\mathbb{P}_{\boldsymbol{p}}\left(X\leftrightarrow Y\right)\right]\geqslant E\left[\mathbb{P}_{\boldsymbol{p}}\left(X\leftrightarrow Y'\right)\right]
\]
\end{prop}

Finally, we give a simple upper bound on the difference between $\mathbb{P}_{p}\left(u\leftrightarrow v\right)$
and $\mathbb{P}_{p}\left(u\leftrightarrow v'\right)$.
\begin{prop}
\label{prop:BunkbedUpperBound}For any bunked graph $G$ and any vector
of percolation parameter $\boldsymbol{p}$, 
\begin{equation}
\left|\mathbb{P}_{\boldsymbol{p}}\left(u\leftrightarrow v\right)-\mathbb{P}_{\boldsymbol{p}}\left(u\leftrightarrow v'\right)\right|\leqslant\mathbb{P}_{\boldsymbol{p}}\left(u\not\leftrightarrow u'\cap v\not\leftrightarrow v'\right)\label{eq:upperbound}
\end{equation}
\end{prop}

In the case of the bunkbed graph of $K_{n}$ when the vector of percolation
is constant, $\forall e,p_{e}=p$, note that $\eqref{eq:upperbound}$
is bounded by $\mathbb{P}_{p}\left(u\not\leftrightarrow u'\right)\leqslant\left(1-p\right)\left(1-p^{3}\right)^{n-1}$.
From now on, we fix $n\in\mathbb{N}$ and $G$ will denote the bunkbed
graph of $K_{n}$.

\section{Proof of Theorem $\ref{thm:BBCTheoremePrincipal}$}

In this section, we prove the main theorem in the following manner. 

First we fix a vertex $u$ on the bottom graph. We define $G_{x,y,z}$
the set of the induced subgraphs of $G$ containing $u$ with $x$
vertices on the bottom graph, $y$ vertices on the top graph, and
$z$ vertices of the bottom graph having their symmetrical in the
subgraph. From this definition $x\geqslant1$, since $u$ is always
in the subgraphs, and $z\leqslant\min\left(x,y\right)$. An example
of these graphs is given in the figure $\ref{fig:BBCFigureQuiDechire}$.

Secondly, we fix a vertex $v$ on the bottom graph. We define $G^{1}$
the set of induced subgraph of $G$ containing $v$ and $G^{2}$ the
set of induced subgraph containing $v'$. 

From these three definitions, we introduce the sets $G_{x,y,z}^{1}=G^{1}\cap G_{x,y,z}$
and $G_{x,y,z}^{2}=G^{2}\cap G_{x,y,z}$. Noting $\left|X\right|$
the cardinality of the set $X$, we introduce the functions $C_{1}:\mathbb{N}^{3}\to\mathbb{N}$,
$C_{2}:\mathbb{N}^{3}\to\mathbb{N}$ and $C_{diff}:\mathbb{N}^{3}\to\mathbb{N}$
defined by:
\begin{align*}
C_{1}\left(x,y,z\right) & =\left|G_{x,y,z}^{1}\right|\\
C_{2}\left(x,y,z\right) & =\left|G_{x,y,z}^{2}\right|\\
C_{diff}\left(x,y,z\right) & =\begin{cases}
C_{1}\left(x,y,z\right)-C_{2}\left(x,y,z\right)\\
\quad+C_{1}\left(y,x,z\right)-C_{2}\left(y,x,z\right) & \mbox{ if }x\neq y\\
C_{1}\left(x,x,z\right)-C_{2}\left(x,x,z\right) & \mbox{ if }x=y
\end{cases}
\end{align*}

Note that, since we considered the bunkbed of a complete graph, any
graph of $G_{x,y,z}$ is isomorph to any other graph of $G_{x,y,z}$
, so we introduce the function $P\colon\mathbb{N}^{3}\to\mathbb{N}$
which associates to a triplet $\left(x,y,z\right)$ the probability
for a graph of $G_{x,y,z}$ to be connected. 

We call the main component of a configuration $\omega$, the set of
vertices connected by an open path to $u$. Since the main component
is a spanning subgraph of $G'\in G_{x,y,z}$, it has to be connected
and isolated from the rest; thus we introduce the number edges that
have to be closed on the boundary of the main component with the function
$B:\mathbb{N}^{3}\to\mathbb{N}$. 

Note that $C_{diff}$, $P$ and $B$ are symmetrical in their two
first coordinates unlike the functions $C_{1}$ and $C_{2}$.

If the main component contains $v$ (resp. $v'$), then there exist
$x,y,z$ such that it is a spanning subgraph of a $G'\in G_{x,y,z}^{1}$
(resp. $G'\in G_{x,y,z}^{2}$). To quantify $\mathbb{P}_{p}\left(u\leftrightarrow v\right)$,
we will classify the configurations $\omega$ according to which set
$G_{x,y,z}^{1}$ (resp. $G_{x,y,z}^{2}$) the main component is a
spanning subgraph. 

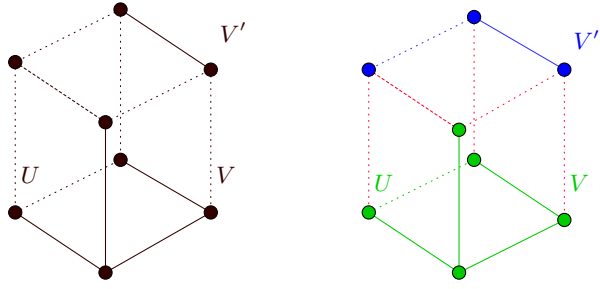
\begin{figure}[!h]
\definecolor{ttqqqq}{rgb}{0.2,0.,0.}
\definecolor{ttttff}{rgb}{0.2,0.2,1.} 
\definecolor{ffqqtt}{rgb}{1.,0.,0.2} 
\definecolor{qqqqff}{rgb}{0.,0.,1.} 
\definecolor{qqccqq}{rgb}{0.,0.8,0.} 

\begin{center}
\begin{tikzpicture}[line cap=round,line join=round,>=triangle 45,x=1.0cm,y=1.0cm] \clip(0.4,0.3) rectangle (10,6); 
\draw [dotted,color=ffqqtt] (5.9,3.6)-- (5.9,1.7); 
\draw [color=qqccqq] (7.1,2.8)-- (7.1,0.9); 
\draw [dotted,color=ffqqtt] (8.5,3.6)-- (8.5,1.6); 
\draw [dotted,color=ffqqtt] (7.3,4.3)-- (7.3,2.4); 
\draw [dotted,color=ttttff] (7.3,4.3)-- (5.9,3.6);
\draw [dash pattern=on 1pt off 1pt,color=ffqqtt] (7.1,2.8)-- (5.9,3.6); 
\draw [color=ttttff] (7.3,4.3)-- (8.5,3.6); 
\draw [dotted,color=ffqqtt] (7.1,2.8)-- (8.5,3.6); 
\draw [color=qqccqq] (8.5,1.6)-- (7.1,0.9); 
\draw [color=qqccqq] (7.1,0.9)-- (5.9,1.7); 
\draw [dotted,color=qqccqq] (5.9,1.7)-- (7.3,2.4); 
\draw [color=qqccqq] (7.3,2.4)-- (8.5,1.6);
\draw [dotted,color=ttqqqq] (1.2,3.7)-- (1.2,1.7); 
\draw [color=ttqqqq] (2.4,2.9)-- (2.4,0.9); 
\draw [dotted,color=ttqqqq] (3.8,3.6)-- (3.8,1.7); 
\draw [dotted,color=ttqqqq] (2.6,4.4)-- (2.6,2.4); 
\draw [dotted,color=ttqqqq] (2.6,4.4)-- (1.2,3.7); 
\draw [dash pattern=on 1pt off 1pt,color=ttqqqq] (2.4,2.9)-- (1.2,3.7); 
\draw [color=ttqqqq] (2.6,4.4)-- (3.8,3.6); 
\draw [dotted,color=ttqqqq] (2.4,2.9)-- (3.8,3.6); 
\draw [color=ttqqqq] (3.8,1.7)-- (2.4,0.9); 
\draw [color=ttqqqq] (2.4,0.9)-- (1.2,1.7); 
\draw [dotted,color=ttqqqq] (1.2,1.7)-- (2.6,2.4); 
\draw [color=ttqqqq] (2.6,2.4)-- (3.8,1.7);

\begin{scriptsize} 
	\draw [fill=qqccqq] (5.9,1.7) circle (2.5pt); 
	\draw[color=qqccqq] (6.1,2.1) node {$U$}; 
	\draw [fill=qqccqq] (7.1,0.9) circle (2.5pt); 
	\draw [fill=qqccqq] (8.5,1.6) circle (2.5pt); 
	\draw[color=qqccqq] (8.7,2.1) node {$V$}; 
	\draw [fill=qqccqq] (7.3,2.4) circle (2.5pt); 
	\draw [fill=qqqqff] (5.9,3.6) circle (2.5pt); 
	\draw [fill=qqccqq] (7.1,2.8) circle (2.5pt); 
	\draw [fill=qqqqff] (8.5,3.6) circle (2.5pt);
	\draw[color=qqqqff] (8.8,4.0) node {$V'$}; 
	\draw [fill=qqqqff] (7.3,4.3) circle (2.5pt); 
	\draw [fill=ttqqqq] (1.2,1.7) circle (2.5pt); 
	\draw[color=ttqqqq] (1.4,2.2) node {$U$}; 
	\draw [fill=ttqqqq] (2.4,0.9) circle (2.5pt); 
	\draw [fill=ttqqqq] (3.8,1.7) circle (2.5pt); 
	\draw[color=ttqqqq] (4.0,2.2) node {$V$}; 
	\draw [fill=ttqqqq] (2.6,2.4) circle (2.5pt); 
	\draw [fill=ttqqqq] (1.2,3.7) circle (2.5pt); 
	\draw [fill=ttqqqq] (2.4,2.9) circle (2.5pt); 
	\draw [fill=ttqqqq] (3.8,3.6) circle (2.5pt); 
	\draw[color=ttqqqq] (4.1,4.1) node {$V'$}; 
	\draw [fill=ttqqqq] (2.6,4.4) circle (2.5pt);
\end{scriptsize} 
\end{tikzpicture}
\end{center}

\caption{Decomposition of a configuration}

\label{fig:BBCFigureDeDecomposition}
\end{figure}

We illustrate these quantities with the figure $\ref{fig:BBCFigureDeDecomposition}$
where, for the sake of simplicity, we draw the bunkbed graph of a
square. Full edges represent open edges and dotted lines the closed
ones. Green vertices and vertices belong to the main component. Red
edges are the exterior edges that have to be closed. Blue vertices
and edges are the remaining components of the graph. 

Recall that we introduced these functions to decompose the quantities
$\mathbb{P}_{p}\left(u\leftrightarrow v\right)$ and $\mathbb{P}_{p}\left(u\leftrightarrow v'\right)$.
The configurations $\omega$ have as their main components a connected
spanning subgraph of $G'\in G_{x,y,z}$, which are for fixed $x,y,z$
isomorph to each other and whose probability to be connected is $P\left(x,y,z\right)$,
and a number $B\left(x,y,z\right)$ of closed boundary edges. Writing
$MC(\omega)$ the main component of $\omega$, we obtain:
\begin{eqnarray*}
\mathbb{P}_{p}\left(u\leftrightarrow v\right) & = & \sum_{x,y,z}\sum_{\substack{G'=\left(V',E'\right)\in G_{x,y,z}^{1}}
}\sum_{\omega}\mathbb{P}\left(MC\left(\omega\right)=V'\right)\\
 & = & \sum_{x,y,z}\sum_{\substack{G'=\left(V',E'\right)\in G_{x,y,z}^{1}}
}\mathbb{P}\left(G'\text{ is connected}\right)\times\left(1-p\right)^{B\left(x,y,z\right)}\\
 & = & \sum_{x,y,z}C_{1}\left(x,y,z\right)\times P\left(x,y,z\right)\times\left(1-p\right)^{B\left(x,y,z\right)}
\end{eqnarray*}
Likewise, it holds: 
\[
\mathbb{P}_{p}\left(u\leftrightarrow v'\right)=\sum_{x,y,z}C_{2}\left(x,y,z\right)\times P\left(x,y,z\right)\times\left(1-p\right)^{B\left(x,y,z\right)}
\]
Taking the difference between these two last quantities, and reindexing,
we obtain:
\begin{eqnarray}
 &  & \mathbb{P}_{p}\left(u\leftrightarrow v\right)-\mathbb{P}_{p}\left(u\leftrightarrow v'\right)\nonumber \\
 & = & \sum_{x,y,z}\left(P\times\left(1-p\right)^{B}\times\left(C_{1}-C_{2}\right)\right)\left(x,y,z\right)\nonumber \\
 & = & \sum_{z\geqslant0}\sum_{x\geqslant y}\left(P\times\left(1-p\right)^{B}\times C_{diff}\right)\left(x,y,z\right)\nonumber \\
 & = & \sum_{z\geqslant0}\sum_{k\geqslant0}\sum_{i\in\mathbb{N}}\left(P\times\left(1-p\right)^{B}\times C_{diff}\right)\left(k+i,k-i,z\right)\nonumber \\
 &  & \quad+\sum_{z\geqslant0}\sum_{k\geqslant0}\sum_{i\in\mathbb{N}}\left(P\times\left(1-p\right)^{B}\times C_{diff}\right)\left(k+1+i,k-i,z\right)\label{eq:BBCPreuveThmPrincipal}
\end{eqnarray}
The proof of the theorem relies on the two key lemmas:
\begin{lem}
\label{lem:BBCComparaisonI0} Let $p\geqslant1/2$. $\forall\epsilon\in\left\{ 0;1\right\} $
there exists $i_{0}:=i_{0}\left(n\right)$ such that $\forall i<i_{0}$,
we have:\textup{
\begin{eqnarray*}
\left(\left(1-p\right)^{B}\times P\right)\left(k+i+\epsilon,k-i,z\right) & \leqslant & \left(\left(1-p\right)\times P\right)\left(k+i_{0}+\epsilon,k-i_{0},z\right)\\
C_{diff}\left(k+i+\epsilon,k-i,z\right) & < & 0
\end{eqnarray*}
and for all $i\geqslant i_{0}$: 
\begin{eqnarray*}
\left(\left(1-p\right)^{B}\times P\right)\left(k+i+\epsilon,k-i,z\right) & \geqslant & \left(\left(1-p\right)^{B}\times P\right)\left(k+i_{0}+\epsilon,k-i_{0},z\right)\\
C_{diff}\left(k+i+\epsilon,k-i,z\right) & \geqslant & 0
\end{eqnarray*}
}
\end{lem}

\begin{rem}
One may ask if the lower bounds $1/2$ is optimal. Coarse estimates
on the function $P$ have been used to prove this lemma to get a simplified
proof. One can using better estimates and lower the bound. The author
has obtained a bound of $0.42$ using computer calculation.
\end{rem}

\begin{lem}
\label{lem:BBCsommeCDiff} For all $z>0$ and for all $k\geqslant z$,
we have: 
\[
\sum_{i=0}^{k-z}C_{diff}\left(k+i,k-i,z\right)=\sum_{i=0}^{k-z}C_{diff}\left(k+i+1,k-i,z\right)=0
\]
\end{lem}

Once these lemmas stated, we can end the proof of the theorem. We
start by noting that using lemma $\ref{lem:BBCComparaisonI0}$, when
$i>i_{\text{0}}$, the quantity $C_{diff}$ is positive, and we can
lower bound $\left(\left(1-p\right)^{B}\times P\right)\left(k+i,k-i,z\right)$
by $\left(\left(1-p\right)^{B}\times P\right)\left(k+i_{0},k-i_{0},z\right)$,
whereas when $i<i_{0}$, the quantity $C_{diff}$ is negative, we
can upper bound $\left(\left(1-p\right)^{B}\times P\right)\left(k+i,k-i,z\right)$
by $\left(\left(1-p\right)^{B}\times P\right)\left(k+i_{0},k-i_{0},z\right)$.
Finally, when $z=0$, meaning that none of the vertices are connected
to their symmetrical, no configuration connect $u$ to $v'$. In this
way, we get:
\begin{eqnarray*}
 &  & \left(\ref{eq:BBCPreuveThmPrincipal}\right)\\
 & \geqslant & \sum_{z>0}\left(\left(1-p\right)^{B}P\right)\left(k+i_{0},k-i_{0},z\right)\sum_{i=0}^{k-z}C_{diff}\left(k+i,k-i,z\right)\\
 &  & \quad+\sum_{z>0}\left(\left(1-p\right)^{B}P\right)\left(k+i_{0}+1,k-i_{0},z\right)\sum_{i=0}^{k-z}C_{diff}\left(k+i+1,k-i,z\right)\\
 &  & =0
\end{eqnarray*}
Where the last inequality is obtained using lemma $\ref{lem:BBCsommeCDiff}$
concluding the proof of the main theorem.

\begin{flushright} $\square$ \end{flushright}

\section{Proof of the Technical Lemmas}

We start by proving the two technical lemmas used in the proof of
the main theorem. Recall that we considered the bunkbed of the complete
graph with $n$ vertices. We start by proving lemma $\ref{lem:BBCComparaisonI0}$
using three intermediates lemmas. 
\begin{lem}
\label{lem:BBCComparaisonE} Let $x,x',y,y',z\in\left[0;n\right]\cap\mathbb{N}$
such that $x+y=x'+y'$ and $z\leqslant\min\left(x,x',y,y'\right)$.
If $\left|x-y\right|>\left|x'-y'\right|$ then $B\left(x,y,z\right)\geqslant B\left(x',y',z\right)+2$
\end{lem}

\begin{proof}
A simple calculation gives for all $x,y,z$ :
\[
B\left(x,y,z\right)=\left(x+y\right)n-x^{2}+x-y^{2}+y-2z
\]
Then, to prove the result, it is enough to show the result for $x=x'+1\geqslant y=y'-1$.
Using the equality above, we obtain:
\[
B\left(x+1,y-1,z\right)-B\left(x,y,z\right)=-2x+2y-2\leqslant0
\]
\end{proof}
\begin{lem}
\label{lem:BBCComparaisonGC}Let $p\geqslant1/2$ and let $x,x',y,y',z\in\left[0;n\right]\cap\mathbb{N}$
such that $x+y=x'+y'$ and $z\leqslant\min\left(x,x',y,y'\right)$.
If $\left|x-y\right|>\left|x'-y'\right|$ then $\left(\left(1-p\right)^{B}P\right)\left(x,y,z\right)\geqslant\left(\left(1-p\right)^{B}P\right)\left(x',y',z\right)$
\end{lem}

\begin{proof}
Again, to prove this lemma and according to the previous one, it is
enough to prove for $x>y\geqslant z$: 
\begin{equation}
P\left(x+1,y,z\right)\geqslant\left(1-p\right)^{2}P\left(x,y+1,z\right)\label{eq:recurrenceAProuver}
\end{equation}
As an upper bound for $P\left(x,y,z\right)$, we use the probability
that there is at least one vertical edge is open, \emph{i.e.} that
$P\left(x,y,z\right)\leqslant1-\left(1-p\right)^{z}$. Then we lower
$P\left(x,y,z\right)$ by the event that the upper graph is connected
as well as the lower graph and at least one vertical edge open:
\begin{equation}
P\left(x,y,z\right)\geqslant P\left(x,0,0\right)\times P\left(0,y,0\right)\times\left(1-\left(1-p\right)^{z}\right)\label{eq:comparaisonGCLowerBound}
\end{equation}
Furthermore, one can use as a lower bound the probability for a complete
graph to be disconnected is greater that $1-$ the sum of the probability
of any set of vertices of cardinality less that $n/2$ is connected
and isolated from the rest:
\begin{align*}
P\left(n,0,0\right) & \geqslant1-\sum_{k=1}^{n/2}\binom{n}{k}\left(1-p\right)^{k\left(n-k\right)}\\
 & \geqslant2-\sum_{k=0}^{n/2}\binom{n}{k}\left(1-p\right)^{kn/2}\\
 & =2-\left(1+\left(1-p\right)^{n/2}\right)^{n}
\end{align*}
The function $n\mapsto\left(1+\left(1-p\right)^{n/2}\right)^{n}$
is a decreasing function as long as $n\geqslant2$. Thus, for $n\geqslant10$,
we have that:
\[
P\left(n,0,0\right)\geqslant0.6\geqslant1-p
\]
For $n\leqslant10$, one can lower bound the probability for the complete
graph to be connected with $p=1/2$ and use OEISA001187\footnote{This sequence gives the number of connected graph with $n$ vertices.}
to see that this probability is greater than $1/2$. Plugging these
bounds into $\eqref{eq:recurrenceAProuver}$, one obtains:
\[
\frac{P\left(x+1,y,z\right)}{P\left(x,y+1,z\right)}\geqslant P\left(x,0,0\right)\times P\left(0,y,0\right)\geqslant1/4\geqslant\left(1-p\right)^{2}
\]
which shows the desired result.
\end{proof}
In the following, we will use the following equality valid for all
$k\in\mathbb{N}$~:
\begin{equation}
\frac{1}{\left(x-k\right)!}\mathds{1}_{x\geqslant k}=\frac{1}{x!}\prod_{i=0}^{k}\left(x-i\right)\label{eq:BBCFactorielleIndicatrice}
\end{equation}
Recall that we fixed three vertices $u,v$ and $v'$ (the symmetrical
of $v$) of the bunkbed graph such that $u\neq v$ in order to study
the quantity $\mathbb{P}\left(u\leftrightarrow v\right)-\mathbb{P}\left(u\leftrightarrow v'\right)$.
\begin{prop}
\label{pro:BBCPropositionC1}For all $x,y\geqslant z\geqslant1$ such
that $x+y-z\leqslant n$
\[
C_{1}(x,y,z)=\frac{\left(n-2\right)!x\left(x-1\right)}{\left(x-z\right)!z!\left(n-x-y+z\right)!\left(y-z\right)!}
\]
\end{prop}

\begin{proof}
First of all, if $z=0$ and $y>0$, then a graph of $G_{x,y,z}$ can't
be connected. Moreover, $u$ and $v$ must be in the set of vertices
of the graphes of $G_{x,y,z}^{1}$, thus $x$ has to be greater than
2 else $G_{x,y,z}^{1}$ is an empty set. Then, we obtain a graph of
$G_{x,y,z}^{1}$, we choose $x-2$ vertices on the bottom graph among
the $n-2$ vertices left, $z$ vertices above the $x$ vertices previously
chosen, and finally $y-z$ vertices among the $n-x$ vertices left.
We can write:
\[
C_{1}(x,y,z)=\binom{n-2}{x-2}\times\binom{x}{z}\times\binom{n-x}{y-z}\times\mathds{1}_{x\geqslant2}
\]
We can finally conclude using $\eqref{eq:BBCFactorielleIndicatrice}$
:
\begin{eqnarray*}
C_{1}(x,y,z) & = & \frac{\left(n-2\right)!}{\left(n-x\right)!\left(x-2\right)!}\times\frac{x!}{\left(x-z\right)!z!}\times\frac{\left(n-x\right)!}{\left(n-x-y+z\right)!\left(y-z\right)!}\mathds{1}_{x\geqslant2}\\
 & = & \frac{\left(n-2\right)!x\left(x-1\right)}{\left(x-z\right)!z!\left(n-x-y+z\right)!\left(y-z\right)!}
\end{eqnarray*}
\end{proof}
\begin{prop}
\label{pro:BBCPropositionC2}For all $x,y\geqslant z\geqslant1$ such
that $x+y-z\leqslant n$:
\[
C_{2}\left(x,y,z\right)=\frac{\left(n-2\right)!\left(xy-z\right)}{\left(x-z\right)!z!\left(n-x-y+z\right)!\left(y-z\right)!}
\]
\end{prop}

\begin{proof}
First of all, note that a graph of $G_{x,y,z}^{2}$, $u$ and $v'$
has to belong to the set of vertices. Thus, to enumerate the number
of graph $G=\left(V,E\right)$ of $G_{x,y,z}^{2}$, we distinguish
4 différent cases: either $v$ and $u'$ belongs to $V$; either $v\in V$
and $u'\notin V$; either $v\notin V$ and $u'\in V$ ; either $u',v'\in V$.
We can write:
\begin{eqnarray*}
 &  & C_{2}\left(x,y,z\right)\\
 &  & =\binom{n-2}{x-2}\binom{x-2}{z-2}\binom{n-x}{y-z}\times\mathds{1}_{x\geqslant2,y\geqslant2,z\geqslant2}\\
 &  & \quad+\binom{n-2}{x-2}\binom{x-2}{z-1}\binom{n-x}{y-z}\times\mathds{1}_{x>\max\left(1,z\right),y<n}\\
 &  & \quad+\binom{n-2}{x-1}\binom{x-1}{z-1}\binom{n-x-1}{y-z-1}\times\mathds{1}_{x<n,y>\max\left(1,z\right)}\\
 &  & \quad+\binom{n-2}{x-1}\binom{x-1}{z}\binom{n-x-1}{y-z-1}\times\mathds{1}_{\max\left(1,z\right)<x<n,\max\left(1,z\right)<y<n}
\end{eqnarray*}
Using $\left(\ref{eq:BBCFactorielleIndicatrice}\right)$, we obtain:
\begin{eqnarray*}
C_{2}\left(x,y,z\right) & = & \frac{\left(n-2\right)!}{\left(x-z\right)!z!\left(n-x-y+z\right)!\left(y-z\right)!}\times z\left(z-1\right)\\
 &  & +\frac{\left(n-2\right)!}{\left(x-z\right)!z!\left(n-x-y+z\right)!\left(y-z\right)!}\times z\left(x-z\right)\\
 &  & +\frac{\left(n-2\right)!}{\left(x-z\right)!z!\left(n-x-y+z\right)!\left(y-z\right)!}\times z\left(y-z\right)\\
 &  & +\frac{\left(n-2\right)!}{\left(x-z\right)!z!\left(n-x-y+z\right)!\left(y-z\right)!}\times\left(x-z\right)\left(y-z\right)\\
 & = & \frac{\left(n-2\right)!}{\left(x-z\right)!z!\left(n-x-y+z\right)!\left(y-z\right)!}\left(xy-z\right)
\end{eqnarray*}
Which is the desired result.
\end{proof}
\begin{lem}
\label{lem:BBCNegativiteCDiff}For all $x,y,z\in\mathbb{N}\backslash\left\{ 0\right\} $,
there exists an $x_{0}:=x_{0}\left(y,z\right)$ such that $y\leqslant x\leqslant x_{0}\Rightarrow C_{diff}\left(x,y,z\right)\leqslant0$.
\end{lem}

\begin{proof}
Using propositions $\ref{pro:BBCPropositionC1}$ and $\ref{pro:BBCPropositionC2}$,
we have that:
\[
C_{1}\left(x,y,z\right)-C_{2}\left(x,y,z\right)=\frac{\left(n-2\right)!\left(x^{2}-x-xy+z\right)}{\left(x-z\right)!z!\left(n-x-y+z\right)!\left(y-z\right)!}
\]
From the definition of $C_{diff}$, we have that $x\geqslant z\geqslant1$:
\[
C_{diff}\left(x,x,z\right)=\frac{\left(n-2\right)!\left(z-x\right)}{\left(x-z\right)!\left(x-z\right)!z!\left(n-2x+z\right)!}
\]
and for all $x,y\geqslant z\geqslant1$:
\begin{equation}
C_{diff}\left(x,y,z\right)=\frac{\left(n-2\right)!\left(x^{2}-2xy+y^{2}-x-y+2z\right)}{\left(x-z\right)!\left(y-z\right)!z!\left(n-x-y+z\right)!}\label{eq:BBCValeurDeCDiff}
\end{equation}
Thus, we obtain: 
\[
C_{diff}\left(x,y,z\right)\leqslant0\Leftrightarrow x\in\left[y+\frac{1-\sqrt{8y-8z+1}}{2};y+\frac{1+\sqrt{8y-8z+1}}{2}\right]
\]
\end{proof}
Combining lemmas $\ref{lem:BBCComparaisonE}$, $\ref{lem:BBCComparaisonGC}$
and $\ref{lem:BBCNegativiteCDiff}$, lemma $\ref{lem:BBCComparaisonI0}$
is shown.

Then we prove lemma $\ref{lem:BBCsommeCDiff}$ by proving two intermediate
lemmas.
\begin{lem}
\label{lem:BBCSommeCDiffNulle_1}For all $k\geqslant z\geqslant1$,
we have the following equality: 
\[
\sum_{i=0}^{k-z}C_{diff}\left(k+i,k-i,z\right)=0
\]
\end{lem}

\begin{proof}
To prove this lemma, it is easier to show that: 
\[
\sum_{i=1}^{k-z}C_{diff}\left(k+i,k-i,z\right)=-C_{diff}\left(k,k,z\right)
\]
Using the argument of the function $C_{diff}$, the 3-tuple $\left(k+i,k-i,z\right)$,
some of the factors$\left(\ref{eq:BBCValeurDeCDiff}\right)$ are independent
of $i$. Indeed, we have:
\[
C_{diff}\left(k+i,k-i,z\right)=\frac{4i^{2}-2k+2z}{\left(k+i-z\right)!\left(k-i-z\right)!}\times\frac{\left(n-2\right)!}{z!\left(n-2k+z\right)!}
\]
So to prove the lemma, it is enough to show that:
\[
\sum_{i=1}^{k-z}\frac{4i^{2}-2k+2z}{\left(k+i-z\right)!\left(k-i-z\right)!}=\frac{k-z}{\left(k-z\right)!\left(k-z\right)!}
\]
Note that:
\begin{eqnarray*}
\frac{k-z}{\left(k-z\right)!\left(k-z\right)!} & = & \frac{4-2k+2z}{\left(k+1-z\right)!\left(k-1-z\right)!}\\
 &  & +\frac{3\left(k+2-z\right)}{\left(k+2-z\right)!\left(k-2-z\right)!}\mathds{1}_{k-z\geqslant2}
\end{eqnarray*}
and for all $k-z>i\geqslant2$:
\begin{eqnarray*}
\frac{\left(2i-1\right)\left(k+i-z\right)}{\left(k+i-z\right)!\left(k-i-z\right)!} & = & \frac{4i^{2}-2k+2z}{\left(k+i-z\right)!\left(k-i-z\right)!}\\
 &  & +\frac{\left(2i+1\right)\left(k+i+1-z\right)}{\left(k+i+1-z\right)!\left(k-i-1-z\right)!}
\end{eqnarray*}
And when $i=k-z$, then:
\[
\frac{\left(2i-1\right)\left(k+i-z\right)}{\left(k+i-z\right)!\left(k-i-z\right)!}=\frac{4\left(k-z\right)^{2}-2k+2z}{\left(2k-2z\right)!}
\]
Which concludes the proof.
\end{proof}
\begin{lem}
\label{lem:BBCSommeCDiffNulle_2}For all $k\geqslant z\geqslant1$,
the following equality holds: 
\[
\sum_{i=0}^{k-z}C_{diff}\left(k+i+1,k-i,z\right)=0
\]
\end{lem}

\begin{proof}
Following the proof of the lemma $\ref{lem:BBCSommeCDiffNulle_1}$,
it is enough to show:
\[
\sum_{i=1}^{k-z}\frac{2i^{2}+2i-k+z}{\left(k+i+1-z\right)!\left(k-i-z\right)!}=\frac{k-z}{\left(k+1-z\right)!\left(k-z\right)!}
\]
Since:
\begin{eqnarray*}
\frac{k-z}{\left(k+1-z\right)!\left(k-z\right)!} & = & \frac{4-k+z}{\left(k+2-z\right)!\left(k-1-z\right)!}\\
 &  & +\frac{2\left(k-z+3\right)}{\left(k+3-z\right)!\left(k-2-z\right)!}\mathds{1}_{k-z\geqslant2}
\end{eqnarray*}
and for all $k-z>i\geqslant2$, we have:
\begin{eqnarray*}
\frac{i\left(k+i+1-z\right)}{\left(k+i-z\right)!\left(k-i-z\right)!} & = & \frac{2i^{2}+2i-k-z}{\left(k+i-z\right)!\left(k-i-z\right)!}\\
 &  & +\frac{\left(i+1\right)\left(k+i+2-z\right)}{\left(k+i+1-z\right)!\left(k-i-1-z\right)!}
\end{eqnarray*}
and when $i=k-z$,
\[
\frac{i\left(k+i+1-z\right)}{\left(k+i+1-z\right)!\left(k-i-1-z\right)!}=\frac{2\left(k-z\right)^{2}+2\left(k-z\right)-k+z}{\left(2k-2z+1\right)!}
\]
Which concludes the proof of the lemma.
\end{proof}
The proof of the lemma $\ref{lem:BBCsommeCDiff}$ is the combination
of the lemmas $\ref{lem:BBCSommeCDiffNulle_1}$ and $\ref{lem:BBCSommeCDiffNulle_2}$.

\section{Proof of Auxiliary Results and Remarks }

\subsection{Proof of the Proposition $\ref{prop:BunkbedMoyenne}$}

Recall that the bunkbed conjecture can be reformulated with a set
of parameter of percolation constrained as explained in the second
section. In this context, we prove that the bunkbed conjecture is
verified in mean by considering two independent random variables $X$
and $Y$ identically distributed on the vertices of the bottom graph.
\begin{proof}[Proof of Proposition $\ref{prop:BunkbedMoyenne}$]
We show a slightly stronger result, for all configuration $\omega$,
\[
E\left[\mathds{1}_{X\overset{\omega}{\leftrightarrow}Y}+\mathds{1}_{X'\overset{\omega}{\leftrightarrow}Y'}\right]\geqslant E\left[\mathds{1}_{X\overset{\omega}{\leftrightarrow}Y'}+\mathds{1}_{X'\overset{\omega}{\leftrightarrow}Y}\right]
\]
Given a configuration $\omega$, we look at all of its clusters. For
all $x$, we note $A\left(x\right)$ the cluster containing $x$ intersected
with the set of the vertices of the bottom and $B\left(x\right)$
the cluster containing $x$ intersected with the set of the vertices
of the top graph so that: 
\begin{eqnarray}
 &  & E\left[\mathds{1}_{X\overset{\omega}{\leftrightarrow}Y}+\mathds{1}_{X'\overset{\omega}{\leftrightarrow}Y'}\right]-E\left[\mathds{1}_{X\overset{\omega}{\leftrightarrow}Y'}+\mathds{1}_{X'\overset{\omega}{\leftrightarrow}Y}\right]\nonumber \\
 & = & \sum_{x}P\left(X=x\right)\left[\sum_{y\in A\left(x\right)}P\left(Y=y\right)-\sum_{y'\in B\left(x\right)}P\left(Y=y'\right)\right]\nonumber \\
 &  & +\sum_{x'}P\left(X=x'\right)\left[\sum_{y\in B\left(x'\right)}P\left(Y=y'\right)-\sum_{y'\in A\left(x\right)}P\left(Y=y'\right)\right]\nonumber \\
 & = & \sum_{x}P\left(X=x\right)\left[P\left(Y\in A\left(x\right)\right)-P\left(Y\in B\left(x\right)\right)\right]\nonumber \\
 &  & +\sum_{x'}P\left(X=x'\right)\left[P\left(Y\in B\left(x'\right)\right)-P\left(Y\in A\left(x'\right)\right)\right]\label{eq:BBCResultatMoyenE1}
\end{eqnarray}
Then, we sum over the different clusters $C\in C\left(\omega\right)$
instead of the vertices and we note for each cluster a representative
$x_{0}\coloneqq x_{0}\left(C\right)$, so we get: 
\begin{eqnarray*}
\left(\ref{eq:BBCResultatMoyenE1}\right) & = & \sum_{C\in C\left(\omega\right)}\sum_{x\in A\left(x_{0}\right)}P\left(X=x\right)\left[P\left(Y\in A\left(x\right)\right)-P\left(Y\in B\left(x\right)\right)\right]\\
 &  & +\sum_{x\in B\left(x_{0}\right)}P\left(X=x'\right)\left[P\left(Y\in B\left(x'\right)\right)-P\left(Y\in A\left(x'\right)\right)\right]\\
 & = & \sum_{C\in C\left(\omega\right)}\left[P\left(X\in A\left(x_{0}\right)\right)-P\left(X\in B\left(x_{0}\right)\right)\right]\\
 &  & \qquad\qquad\times\left[P\left(Y\in A\left(x_{0}\right)\right)-P\left(Y\in B\left(x_{0}\right)\right)\right]\\
 & = & \sum_{C\in C\left(\omega\right)}\left[P\left(X\in A\left(x_{0}\right)\right)-P\left(X\in B\left(x_{0}\right)\right)\right]^{2}\geqslant0
\end{eqnarray*}
Which proves the result.
\end{proof}

\subsection{Proof of Proposition $\ref{prop:BunkbedUpperBound}$}

The goal of the conjecture is to lower bound the quantity $\mathbb{P}_{p}\left(u\leftrightarrow v\right)-\mathbb{P}_{p}\left(u\leftrightarrow v'\right)$
by $0$. We give a simple upper bound by proving Proposition $\ref{prop:BunkbedUpperBound}$.
\begin{proof}
Note that: 
\begin{eqnarray}
 &  & \mathbb{P}_{p}\left(u\leftrightarrow v\right)\nonumber \\
 & = & \mathbb{P}_{p}\left(u\leftrightarrow v\cap v\leftrightarrow v'\right)+\mathbb{P}_{p}\left(u\leftrightarrow v\cap v\not\leftrightarrow v'\cap u\leftrightarrow u'\right)\nonumber \\
 &  & \qquad+\mathbb{P}_{p}\left(u\leftrightarrow v\cap v\not\leftrightarrow v'\cap u\not\leftrightarrow u'\right)\label{eq:BBCMajorationE1}\\
 &  & \mathbb{P}_{p}\left(u\leftrightarrow v'\right)\nonumber \\
 & = & \mathbb{P}_{p}\left(u\leftrightarrow v'\cap v\leftrightarrow v'\right)+\mathbb{P}_{p}\left(u\leftrightarrow v'\cap v\not\leftrightarrow v'\cap u\leftrightarrow u'\right)\nonumber \\
 &  & \qquad+\mathbb{P}_{p}\left(u\leftrightarrow v'\cap v\not\leftrightarrow v'\cap u\not\leftrightarrow u'\right)\label{eq:BBCMajorationE2}
\end{eqnarray}
The first term of the right-hand of $\eqref{eq:BBCMajorationE1}$
and $\eqref{eq:BBCMajorationE2}$ are obviously equal as well as the
second member by an argument of symmetry. Thus we get: 
\begin{eqnarray*}
 &  & \mathbb{P}_{p}\left(u\leftrightarrow v\right)-\mathbb{P}_{p}\left(u\leftrightarrow v'\right)\\
 & = & \mathbb{P}_{p}\left(u\leftrightarrow v\cap v\not\leftrightarrow v'\cap u\not\leftrightarrow u'\right)-\mathbb{P}_{p}\left(u\leftrightarrow v'\cap v\not\leftrightarrow v'\cap u\not\leftrightarrow u'\right)\\
 & = & \left(\mathbb{P}_{p}\left(u\leftrightarrow v|v\not\leftrightarrow v'\cap u\not\leftrightarrow u'\right)-\mathbb{P}_{p}\left(u\leftrightarrow v'|v\not\leftrightarrow v'\cap u\not\leftrightarrow u'\right)\right)\\
 &  & \quad\times\mathbb{P}_{p}\left(v\not\leftrightarrow v'\cap u\not\leftrightarrow u'\right)\\
 & \geqslant & -\mathbb{P}_{p}\left(v\not\leftrightarrow v'\cap u\not\leftrightarrow u'\right)
\end{eqnarray*}
We proceed in the same way to lower bound $\mathbb{P}_{p}\left(u\leftrightarrow v'\right)-\mathbb{P}_{p}\left(u\leftrightarrow v\right)$
and obtain the desired result.
\end{proof}

\subsection{Trivial Case: the Line Segment}

When we consider the bunkbed graph of the line segment, it is possible
to quantify the difference of $\mathbb{P}_{p}\left(u\leftrightarrow v\right)-\mathbb{P}_{p}\left(u\leftrightarrow v'\right)$.
Consider the line segment graph with $n$ vertices and consider its
bunkbed graph $G$. Note $u$ the point $\left(1,0\right)$ and $v$
the point $\left(n,0\right)$. By symmetry, if the edge $\left\{ \left(i,0\right),\left(i,1\right)\right\} $
is open then $\mathbb{P}\left(\left(i,0\right)\leftrightarrow v\right)=\mathbb{P}\left(\left(i,1\right)\leftrightarrow v\right)=\mathbb{P}\left(\left(i,0\right)\leftrightarrow v'\right)$.
So, noting $\tau$ the index of the first open vertical edge (starting
from the left), and by convention, $\tau=0$ if none are opened, we
get: 
\[
\mathbb{P}\left(u\leftrightarrow v\right)=\sum_{i=1}^{n}\mathbb{P}\left(u\leftrightarrow\left(i,0\right)\cap\left(i,0\right)\leftrightarrow v\cap\tau=i\right)+\mathbb{P}\left(u\leftrightarrow v\cap\tau=0\right)
\]
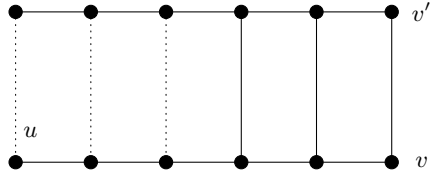
\begin{figure}[!h]
\begin{center}\begin{tikzpicture}[line cap=round,line join=round,>=triangle 45,x=1.0cm,y=1.0cm] \clip(1.68,1.6) rectangle (7.7,4.7); \draw (2.,2.)-- (3.,2.); \draw (3.,2.)-- (4.,2.); \draw (4.,2.)-- (5.,2.); \draw (5.,2.)-- (6.,2.); \draw (6.,2.)-- (7.,2.); \draw (2.,4.)-- (3.,4.); \draw (3.,4.)-- (4.,4.); \draw (4.,4.)-- (5.,4.); \draw (5.,4.)-- (6.,4.); \draw (6.,4.)-- (7.,4.); \draw (5.,4.)-- (5.,2.); \draw (6.,4.)-- (6.,2.); \draw (7.,4.)-- (7.,2.); \draw [dotted] (2.,4.)-- (2.,2.); \draw [dotted] (3.,4.)-- (3.,2.); \draw [dotted] (4.,4.)-- (4.,2.); 
\begin{scriptsize}
\draw [fill=black] (2.,2.) circle (2.5pt);
\draw[color=black] (2.2,2.4) node {$u$}; 
\draw [fill=black] (3.,2.) circle (2.5pt); \draw [fill=black] (4.,2.) circle (2.5pt); \draw [fill=black] (5.,2.) circle (2.5pt); \draw [fill=black] (6.,2.) circle (2.5pt); \draw [fill=black] (7.,2.) circle (2.5pt);
\draw[color=black] (7.4,2) node {$v$}; 
\draw [fill=black] (2.,4.) circle (2.5pt); \draw [fill=black] (3.,4.) circle (2.5pt); \draw [fill=black] (4.,4.) circle (2.5pt); \draw [fill=black] (5.,4.) circle (2.5pt); \draw [fill=black] (6.,4.) circle (2.5pt);
\draw [fill=black] (7.,4.) circle (2.5pt);
\draw[color=black] (7.4,4) node {$v'$}; 
\end{scriptsize} \end{tikzpicture}\end{center}\caption{\label{fig:BBCechelletau}Case where $\tau=4$ and $n=6$}

\end{figure}
Or, the part before the vertical edge $e_{v_{i}}$ is independent
of the part after, conditionally to $\tau=i$ (see figure $\ref{fig:BBCechelletau}$),
which gives:
\begin{eqnarray*}
 &  & \mathbb{P}\left(u\leftrightarrow v\right)\\
 & = & \sum_{i=1}^{n}\mathbb{P}\left(u\leftrightarrow\left(i,0\right)|\tau=i\right)\mathbb{P}\left(\left(i,0\right)\leftrightarrow v|\tau=i\right)\mathbb{P}\left(\tau=i\right)\\
 &  & \qquad\qquad+\mathbb{P}\left(u\leftrightarrow v\cap\tau=0\right)\\
 & = & \sum_{i=1}^{n}\mathbb{P}\left(u\leftrightarrow\left(i,0\right)|\tau=i\right)\mathbb{P}\left(\left(i,0\right)\leftrightarrow v'|\tau=i\right)\mathbb{P}\left(\tau=i\right)\\
 &  & \qquad\qquad+\mathbb{P}\left(u\leftrightarrow v\cap\tau=0\right)\\
 & = & \mathbb{P}\left(u\leftrightarrow v'\right)+\mathbb{P}\left(u\leftrightarrow v\cap\tau=0\right)
\end{eqnarray*}
The difference of the probability is therefore given by: 
\[
\mathbb{P}\left(u\leftrightarrow v\cap\tau=0\right)=\prod_{i=1}^{n-1}\left(1-p_{v_{i}}\right)p_{\left\{ \left(i,0\right)\left(i+1,0\right)\right\} }
\]
We can remark that in this case, the difference of probabilities is
the case where all the vertical edges are closed and $u$ is connected
to $v$. In the proof of the theorem $\ref{thm:BBCTheoremePrincipal}$,
we showed that the difference is strictly greater than this case (this
can be seen in the case of the bunkbed graph of the triangle). In
conclusion, when we consider $p=1/2$ in the theorem $\ref{thm:BBCTheoremePrincipal}$,
it indicates that it would be a difficult task to build a surjection
between the configurations connecting $u$ to $v$ and those connecting
$u$ to $v'$. A second argument supporting this conclusion is the
change of the sign intervening in the quantity $C_{diff}$.

\bibliographystyle{plain}
\bibliography{bibtex_these-Bunkbed}

\end{document}